\documentclass[12pt]{amsart}
\usepackage{a4wide}
\usepackage{float,epsfig,hyperref,graphicx,graphics,mathrsfs,dsfont}
\usepackage{color,cite,epsfig,epstopdf,amsmath,amssymb,amsfonts}
\usepackage{pst-node}
\usepackage{tikz-cd}
\theoremstyle{definition}
\newtheorem{definition}{Definition}[subsection]
\newtheorem{example}{Example}[subsection]
\newtheorem{theorem}{Theorem}[subsection]

\newtheorem{proposition}{Proposition}[subsection]
\newtheorem{lemma}{Lemma}[subsection]
\newtheorem{note}{Note}[subsection]

\newtheorem{remark}{Remark}[subsection]
\title{FRACTAL INTERPOLATION ON THE REAL PROJECTIVE PLANE}

\author{A. Hossain}
\address{Department of Mathematics, Presidency University, 86/1, College Street, Kolkata, 700 073, West Bengal, India}
\curraddr{}
\email{hossain4791@gmail.com}

\author{Md. N. Akhtar}
\address{Department of Mathematics, Presidency University, 86/1, College Street, Kolkata, 700 073, West Bengal, India}
\email{nasim.iitm@gmail.com}
\thanks{The authors thank Akash Banerjee for many helpful discussions to get the figures}
\author{M. A. Navascu\'{e}s}
\address{Departamento de Matem\'{a}tica Aplicada
	Universidad de Zaragoza, Spain}
\email{manavas@unizar.es}
\begin{document}
	\pagestyle{plain}
	\maketitle
\begin{abstract}
Formerly the geometry was based on shapes, but since the last centuries this founding mathematical science deals with transformations, projections and mappings. Projective geometry identifies a line with a single point, like the perspective on the horizon line and, due to this fact, it requires a restructuring of the real mathematical and numerical analysis. In particular, the problem of interpolating data must be refocused. In this paper we define a linear structure along with a metric on a projective space, and prove that the space thus constructed is complete. Then we consider an iterated function system giving rise to a fractal interpolation function of a set of data.
\end{abstract}

	\section{Introduction}
\subsection{Background}
The fractal features describe closely the properties of natural phenomenons. For this reason, the interest in the  mathematical field of the fractal geometry increases rapidly. New procedures of fractal analysis are developed and these procedures are proving their usefulness in real systems in various fields such as informatics \cite{FGN,ficbymf}, engineering \cite{FG,fmss,Scfstr}, medical screening \cite{Ftipb}, biology \cite{Ipbbfa},  cosmology \cite{tfsu}, etc. Also, in dimensions theory estimation of the fractal dimension which may be non-integer value has various application in geometry \cite{bdoaff,bdoaffwvss,Akhtar2022}, has a huge usefulness in fractal geometry.\\
In mathematics, an iterated function system (IFS) is a method of constructing fractals. A fractal interpolation function (FIF) can be considered as a continuous function that interpolates some specific data points and whose graph is the attractor (a fractal set) of an IFS. Barnsley \cite{Ftmapp}, introduced the concept of fractal interpolation  function and it has been widely used in many scientific applications like approximation theory (to approximate discrete sequences of data), image compression, computer graphics, etc. since then. For more details interested readers may consult the references \cite{ficbymf,Uifsmds}. Massopust \cite{prmfs}, presented the construction of self-affine fractal interpolation surfaces (FISs) on a simplex. Navascues \cite{ftsbyman}, constructed a non-self-affine fractal interpolation function as perturbation of any continuous function on a compact set. A rich development in the approximation theory  using non-affine fractal functions can be found in \cite{fpsbyman,fbflps,fabynavascus,ftsbyman,cofsbynasim} and references therein. Vince \cite{mifsonps}, introduced the IFS consisting of M\"{o}bius transformations on the extended complex plane or equivalently on the Riemann sphere. Most of the authors discussed about the FIFs on the Euclidean space \cite{Ftmapp,pbld,bfifog}. Recently, Barnsley et al. \cite{pifs}, introduced the concept of projective IFS on a real projective space. There, the authors  characterized when a projective IFS has an attractor and established the result that a projective IFS has at most one attractor. \\
Projecting a 3D scene onto a 2D image is one of the fundamental issues in 3D computer graphics. In this regard to focus computer vision in general, and especially image formation in particular, projective geometry works as a mathematical framework. Many significant progress has been made in problems as computer vision by applying tools from the classical projective geometry \cite{Opgcv,tdcvagv,pgfiaxv,Pgcvmr,Pgtdcv,mvgcvbyhz}. Projective geometry is usually developed in spaces of a special type, called projective spaces, that are different from the usual affine or Euclidean spaces. A projective space may be viewed as an extension of an Euclidean space, or, more generally, an affine space with points at infinity \cite{ugtopg,apgmt}. Though it has a manifold like structure\cite{aitam}, it is more complicated to develop fractal theory on it. \\
In the literature, a rich development has been made for the constructions of affine FIFs, FISs, non-affine FIFs, and non-affine FISs and their contributions to the field of fractal geometry and approximation theory \cite{prmfs,ftsbyman,pbld,bfifog,BtfabyV,frftabyAk}. But the fractal interpolation theory on the projective space is totally unexplored. The present paper provides a cornerstone of a surprisingly rich mathematical theory associated with the real projective fractal interpolation function (RPFIF). A method is developed to construct a RPFIF for a given data set on the real projective plane $\mathbb{RP}^2\setminus \mathbb{H}_{e_3}$. The advantage of construction of such a RPFIF  is that it is an infinte fractal (in $\mathbb{R}^3$)  consisting of self-affine fractal interpolation functions (which are similar to each other upto contraction) giving thereby a choice of large flexibility in approximating functions.

\subsection{Structure of paper and discussion of results}
In Section~\ref{prelimi}, we introduce  some notation, give basic deﬁnitions of projective space, manifold structure of the projective space, Hausd\"{o}rff metric, attractors and construction of the fractal interpolation functions. In Section~\ref{dcopsahp}, we present a decomposition of the real projective plane $\mathbb{RP}^2\setminus \mathbb{H}_{e_3}$ so that it becomes a vector space over $\mathbb{R}$. A new metric and a norm is introduced on $\mathbb{RP}^2\setminus \mathbb{H}_{e_3}$ to make it a complete normed linear space that provides a setting for the main results of the paper. We define projective interval and projective rectangle which are needed in the construction of the RPFIF in Section~\ref{rpfifdef} and also provide  geometrical structures of these (see Figures~\ref{fig3} and \ref{fig2}).  

In Section~\ref{rpfifdef}, we discuss the construction of a real projective fractal interpolation function for a data set on $\mathbb{RP}^2\setminus \mathbb{H}_{e_3}$. For that a RPIFS is formulated and it is seen that the  maps in the RPIFS contract the projective rectangle while acting on it (see Figure~\ref{prftfns}).
The next theorem is the main result.
\begin{theorem}\label{maint11}
	If $\big\{\mathbb{RP}^2\setminus \mathbb{H}_{e_3}; W_n:~ n=1,2,\ldots,N\big\}$ is a RPIFS, then there exists a fractal function {\bf f} corresponding to it such that the graph of {\bf f} is the attractor of the RPIFS.
\end{theorem}
Figure~\ref{insidevrfif} illustrates the construction of a RPFIF.
Side by side detailed illustrations of the construction of a RPFIF in $\mathbb{RP}^2\setminus \mathbb{H}_{e_3}$ and the corresponding FIF at level $z=z_0$ (or, equivalently in $\mathbb{R}^2$) are provided in this section (see Figure~\ref{initialstep}, \ref{step1}, \ref{step2}, \ref{step3} and \ref{finalstep}). This shows that the RPFIF construction is more inclusive.
In Example~\ref{expoftgotrpfif}, we consider a data set in $\mathbb{RP}^2\setminus \mathbb{H}_{e_3}$ with different scale vectors and see the nature of the graphs of the corresponding RPFIFs respectively.

%
\section{Preliminaries}\label{prelimi}
\subsection{Projective space}
\begin{definition}[Real projective space]
	Given an Euclidean space $\mathbb{R}^{n+1}$, the real projective space associated with $\mathbb{R}^{n+1}$ is the set $\mathbb{RP}^n$ of one dimensional subspaces or (vector) lines in $\mathbb{R}^{n+1}$.
\end{definition}
One can identify $\mathbb{RP}^n$ as the quotient of the set $\mathbb{R}^{n+1}\setminus \{0\}$ of non-zero vectors by the equivalence relation $x\sim y$ if and only if $x=\lambda y$ for some $\lambda\in\mathbb{R}^*$ (non-zero reals). Now, for $x=(x_1,x_2,\ldots,x_{n+1})\in \mathbb{R}^{n+1}\setminus \{0\}$, we denote $(x_1:x_2:\ldots:x_{n+1})$ as the equivalence class  containing $x$. Thus we have a canonical quotient map $\nu:\mathbb{R}^{n+1}\setminus \{0\}\to\mathbb{RP}^n$  that associates to the each non-zero vector $x=(x_1,x_2,\ldots,x_{n+1})\in \mathbb{R}^{n+1}\setminus \{0\}$ to the element $(x_1:x_2:\ldots:x_{n+1})\in \mathbb{RP}^n$. The points $(x_1,x_2,\ldots,x_{n+1})\in \mathbb{R}^{n+1}\setminus \{0\}$ such that $\nu(x_1,x_2,\ldots,x_{n+1})=p$ is referred to as homogeneous coordinates of an element $p\in\mathbb{RP}^n$. 	For more details, interested authors may consult the references \cite{pifs,ugtopg,apgmt}.
Also, one can view $\mathbb{RP}^n$ as a $n$-dimensional manifold
\cite{aitam} with standard atlas
\begin{align*}
	\left\{(U_1,\phi_1),(U_2,\phi_2),\ldots,(U_{n+1},\phi_{n+1})\right\}
\end{align*}
defined as follows. For $k=1,2,\ldots,n+1$, let
\begin{equation*}
	U_k=\left\{(x_1:x_2:\ldots:x_{n+1})\in\mathbb{RP}^n:~x_k\neq 0 \right\}
\end{equation*}
and the chart be given by
\begin{align*}
	\phi_k:U_k\to\mathbb{R}^n,~(x_1:x_2:\ldots:x_{n+1})\to \left(\frac{x_1}{x_k},\frac{x_2}{x_k},\ldots,\frac{x_{k-1}}{x_k},\frac{x_{k+1}}{x_k},\ldots,\frac{x_{n+1}}{x_k}\right).
\end{align*}
This is well defined, as  multiplying  $x_i$ by a non-zero scalar  the quotient does not change.
\begin{definition}[Hyperplane]
	If $p,q\in\mathbb{RP}^n$ have the homogeneous coordinates $(p_1,p_2,\ldots,p_{n+1})$ and $(q_1,q_2,\ldots,q_{n+1})$ respectively, and $\sum_{k=1}^{{n+1}}p_kq_k=0$, then we say that $p$ is orthogonal to $q$, and write $p\perp q$. A {\bf hyperplane} in $\mathbb{RP}^n$ is a set of the form 
	\begin{equation*}
		\mathbb{H}_p=\big\{q\in\mathbb{RP}^n:p\perp q\big\}\subseteq \mathbb{RP}^n
	\end{equation*}
	for some $p\in\mathbb{RP}^n$.
\end{definition}
\begin{definition}[see \cite{pifs}]
	A set $\mathbb{K}\subseteq \mathbb{RP}^n$ is said to {\bf avoid a hyperplane} if there exists a hyperplane $\mathbb{H}_p\subseteq \mathbb{RP}^n$ such that $\mathbb{H}_p\cap \mathbb{K}=\emptyset$.
\end{definition}
\begin{definition}[Line in the real projective space]
	A line in the real projective space is the set of equivalence classes of points in a 2-dimensional subspace of $\mathbb{R}^{n+1}$. In other words,  if $a,b\in\mathbb{RP}^n$ have the homogeneous coordinates $(a_1,a_2,\ldots,a_{n+1})$ and $(b_1,b_2,\ldots,b_{n+1})$ respectively, then the corresponding line $\overline{ab}\subset \mathbb{RP}^n$ has its homogeneous coordinates of the form $\big(ua_1+vb_1,ua_2+vb_2,\ldots,ua_{n+1}+vb_{n+1}\big)$, where $u,v\in\mathbb{R}$, and both are  not zero simultaneously.
\end{definition}
The ``{\bf round}" metric $d_R$ on $\mathbb{RP}^n$ is defined as follows. Each element $x\in \mathbb{RP}^n$ is represented by a line in $\mathbb{R}^{n+1}$ through the origin or by the two points $a_x$ and $b_x$, where this line intersects the unit sphere centered at the origin. Then the round metric is given by $d_R(x,y)=\min\{\Vert a_x-a_y\Vert,\Vert a_x-b_y\Vert\}$, where the norm is the Euclidean norm in $\mathbb{R}^{n+1}$. In term of homogeneous coordinates, the metric is given by 
\begin{equation*}
	d_R(x,y)=\sqrt{2-2\frac{\vert\langle x,y\rangle\vert}{\Vert x \Vert\Vert y \Vert}},
\end{equation*}
where $\langle \cdot,\cdot\rangle $ is the usual Euclidean inner product. The metric space $(\mathbb{RP}^n,d_R)$ is compact \cite{pifs}.

\subsection{Iterated function system}
\begin{definition}[Hausd\"{o}rff metric]
	Let $(X,d)$ be a metric space and $\mathscr{H}(X)$ denotes the space of all non-empty compact subsets of $X$. Then the Hausd\"{o}rff distance between the sets $A$ and $B$ in $\mathscr{H}(X)$, denoted by $h_d$, is defined by
	\begin{displaymath}
		h_d(A, B)=\max \left\{ \sup_{x\in A} \inf_{y\in B}d(x, y),~~ \sup_{y\in B} \inf_{x\in A}d(x, y) \right \}\quad \mbox{for all}~A, B \in \mathscr{H}(X).
	\end{displaymath}
\end{definition}
If $(X,d)$ is a complete metric space, then $\left ( \mathscr{H}(X), h_d \right)$ is also a complete metric space (see \cite{FG,FE}).
\begin{definition}
	Let $(X,d)$ be a complete metric space. If $W_n: X\to X$, $n=1, 2,\ldots, N$, are continuous maps, then  $\mathscr{W}=\{X;W_n:n=1,2,\ldots,N\}$ is called an \textbf{iterated function system} (IFS) (see \cite{FE}). The system is called hyperbolic IFS if each function $W_n:X\rightarrow X$ is contractive with contraction factor $0\leq c_n<1.$ In particular, if $W_n$'s are the projective transformations on the real projective plane, then $\mathscr{W}=\{\mathbb{RP}^2;W_n:n=1,2,\ldots,N\}$ is called a  \textbf{real projective iterated function system} or RPIFS \cite{pifs}.
	
	The $\textbf{Hutchinson operator}$  $W : \mathscr{H}(X) \rightarrow \mathscr{H}(X)$, is defined by
	$$W(B)=\bigcup_{n=1}^N W_n(B) \quad \mbox{for all}~ B \in \mathscr{H}(X).$$
\end{definition}
It is a standard result that if each $W_n$ is a contraction map on $(X,d)$ with contractivity factor $c_n$ for $n=1, 2,\ldots, N$, then the Hutchison operator $W$ is a contraction map with respect to the corresponding Hausd\"{o}rff metric $h_d$ with contractivity factor $c=\max_{n}\{c_n\}$ \cite{FG,FE}. Define $W^0(B) = B$ and let $W^k(B)$ denote the $k$-fold composition of $W$ applied to $B$.
\begin{definition}[see \cite{pifs,mifsonps}]\label{dfnbasin}
	A compact subset $F$ of $(X,d)$ is called an \textbf{attractor}  of the IFS $\mathscr{W}=\{X;W_n:n=1,2,\ldots,N\}$ if
	\begin{enumerate}
		\item $W(F)=F$ and
		\item there exists an  open subset $U$ of $X$ such that $F\subset U$ and
		$$\lim\limits_{k\rightarrow \infty}W^k(B)=F \quad \mbox{for all}~ B\in \mathscr{H}(U),$$
		where the limit is with respect to the Hausd\"{o}rff metric $h_d$ on $\mathscr{H}(X).$	
	\end{enumerate}
\end{definition}
\begin{note}
	The largest open set $U$ in Definition~\ref{dfnbasin} is known as the \textbf{basin} of attraction for the attractor $F$ of the IFS $\mathscr{W}$ and  is denoted by $B(F)$. 
\end{note}
\subsection{Fractal function} Let $I=[a,b]$ and $\Delta$ : $a=x_0<x_1<\cdots<x_N=b$ be a partition of $I$. Let $\{(x_n,y_n)\in\mathbb{R}^2:n=0,1,2,\ldots,N\}$ be the given interpolation points in $\mathbb{R}^2$. Set $I_n=[x_{n-1},x_n]$ for $n=1,2,\ldots,N$. Suppose $L_n:I\rightarrow I_n$ are contraction homeomorphisms such that 
\begin{align}
	\label{lmaps}
	L_n(x_0)&=x_{n-1},~L_n(x_N)=x_n,\\ 
	\vert L_n(x)-L_n(x')\vert&\leq c_n\vert x-x'\vert\quad\mbox{for all}~ x,x'\in I,~\mbox{for some}~0\leq c_n<1.
\end{align}
Further, assume that $F_n:I\times\mathbb{R}\rightarrow\mathbb{R}$ are continuous maps satisfying
\begin{align}\label{fnmaps}
	F_n(x_0,y_0)=y_{n-1},~F_n(x_N,y_N)=y_n,
\end{align}
\begin{align*}
	&\vert F_n(x,y)-F_n(x',y)\vert\leq a_n\vert x-x'\vert\quad\mbox{for all}~ x,x'\in I,\\
	&\vert F_n(x,y)-F_n(x,y')\vert\leq b_n\vert y-y'\vert\quad\mbox{for all}~ y,y'\in\mathbb{R},
\end{align*} 
for some $a_n,b_n\in(-1,1)$. Define the functions $W_n:I\times\mathbb{R}\rightarrow I_n\times\mathbb{R}$ by 
\begin{equation}
	W_n(x,y)=(L_n(x),F_n(x,y)).
\end{equation}
The maps $W_n$ satisfies the join up condition 
\begin{align}
	W_n(x_0,y_0)=(x_{n-1},y_{n-1}),\quad W_n(x_N,y_N)=(x_n,y_n)\quad\mbox{for all}~n=1,2,\ldots,N.
\end{align}
The following is a fundamental result in the theory of fractal interpolation functions.
\begin{theorem}[Barnsley \cite{Ftmapp}]
	Let $C[I]$, the space of all real-valued continuous functions on $I$, be endowed with supremum norm. That is $$\Vert{f}\Vert_{\infty}=\max\left\{\vert f(x)\vert:~x\in I\right\}.$$ Consider the closed subspace 
	\begin{align*}
		C_{y_0,y_N}[I]:=\left\{f\in C[I]:f(x_0)=y_0,~f(x_N)=y_N\right\}.
	\end{align*}
	
	Then the following holds.
	\begin{enumerate}\label{ftoff}	
		\item The IFS $\{(I\times\mathbb{R};W_n):n=1,2,\ldots,N\}$ has unique attractor $G(g)$ which is the graph of a continuous function $g:I\rightarrow \mathbb{R}$ satisfying $g(x_n)=y_n$ for all $n=0,1,\ldots,N$.
		\item The function $g$ is the fixed point of the Read-Bajraktarevic (RB) operator $T:C_{y_0,y_N}[I]\rightarrow C_{y_0,y_N}[I]$ defined by 
		\begin{equation*}
			(Tf)(x)=F_n(L_n^{-1}(x), f(L_n^{-1}(x))),\quad\mbox{for}~x\in I_n;~n=1,2,\ldots,N.
		\end{equation*}
	\end{enumerate}
\end{theorem}
The function $g$ is called the fractal interpolation function (FIF) corresponding to the data set $\{(x_n,y_n)\in\mathbb{R}^2:n=0,1,2,\ldots,N\}$.

\section{Decomposition  of the projective plane which avoids a hyperplane}\label{dcopsahp}
Let $A,B$ be the subsets of $ \mathbb{R}$. Since $\mathbb{R}^2$ can be decomposed as $\mathbb{R}\times\mathbb{R}$,  for any function $f:A\to B$ there is a conventional way to define the $graph(f)=\left\{(x,f(x)):x\in A\right\}$ so that it lies on $\mathbb{R}^2$. But if  one considers a function $f:A\to B$, where $A,B$ are the subsets of $ \mathbb{RP}^n$, then there is no traditional way to define the $graph(f)$ for which it lies on $ \mathbb{RP}^m$ for some $m\in\mathbb{N}$. For this reason to define a function whose graph lies on the projective space, a decomposition  is required. In this section, mainly, we provide a decomposition of the projective plane which avoids a hyperplane. We define a norm which induce a metric on it. Also, projective interval and projective rectangle  are defined and some topological results are proved.

Let $\{e_1=(1,0,0), e_2=(0,1,0),e_3=(0,0,1)\}$ be the canonical basis of $\mathbb{R}^3$. Then $\mathbb{H}_{e_i}$ is the hyperplane perpendicular to $e_i$ for $i=1,2,3$. One may consider the space $\mathbb{RP}^2\setminus \mathbb{H}_{e_i}$ which avoids the hyperplane $\mathbb{H}_{e_i}$  for $i=1,2,3$ respectively. In the sequel, we consider the space $\mathbb{RP}^2\setminus \mathbb{H}_{e_3}$ in particular and define two operations $\oplus$ and $\odot$ as follows. For all $(x:y:z),(x':y':z')\in\mathbb{RP}^2\setminus \mathbb{H}_{e_3}$ and for all $a\in\mathbb{R}$,
\begin{equation}\label{addop}
	(x:y:z)\oplus(x':y':z'):=(xz'+x'z:yz'+y'z:zz')
\end{equation}
and
\begin{equation}
	a\odot(x:y:z):=(ax:ay:z).
\end{equation}
Since $z,z'\neq 0$, implies $zz'\neq 0$. So, $(x:y:z)\oplus(x':y':z')\in \mathbb{RP}^2\setminus \mathbb{H}_{e_3}$ and $a\odot(x:y:z)\in \mathbb{RP}^2\setminus \mathbb{H}_{e_3}$. Also, for non-zero reals $\lambda_1$, $\lambda_2$ and $\lambda$,
\begin{align*}
	(\lambda_1x:\lambda_1y:\lambda_1z)&\oplus(\lambda_2x':\lambda_2y':\lambda_2z')\\
	&=\big(\lambda_1\lambda_2(xz'+x'z):\lambda_1\lambda_2(yz'+y'z):\lambda_1\lambda_2zz'\big)\\
	&=(xz'+x'z:yz'+y'z:zz')\\
	&=(x:y:z)\oplus(x':y':z')	
\end{align*}
and
\begin{align*}
	a\odot(\lambda x:\lambda y:\lambda z)&=(a\lambda x:a\lambda y:\lambda z)\\
	&=(ax:ay:z)=a\odot(x:y:z).
\end{align*}
So, both the operations $\oplus$ and $\odot$ are well defined.
\begin{proposition}
	$\mathbb{RP}^2\setminus \mathbb{H}_{e_3}$ forms a vector space over $\mathbb{R}$ with respect to the above defined operations $\oplus$ and $\odot$.
\end{proposition}
\begin{proof}
	It is easy to verify that $\oplus$ is commutative as well as associative in $\mathbb{RP}^2\setminus \mathbb{H}_{e_3}$. For all $(x_1:y_1:z_1)\in\mathbb{RP}^2\setminus \mathbb{H}_{e_3}$ and $(0:0:z)\in\mathbb{RP}^2\setminus \mathbb{H}_{e_3}$,
	\begin{equation}
		(x_1:y_1:z_1)\oplus(0:0:z)=(x_1z:y_1z:z_1z)=(x_1:y_1:z_1).
	\end{equation}
	Hence $(0:0:z)$ is the zero element in $\mathbb{RP}^2\setminus \mathbb{H}_{e_3}$. Also, for all $(x:y:z)\in\mathbb{RP}^2\setminus \mathbb{H}_{e_3}$,
	\begin{equation}
		(x:y:z)\oplus(-x:-y:z)=(0:0:z).
	\end{equation}
	Therefore, $(-x:-y:z)$ is the additive inverse of $(x:y:z)$ in $\mathbb{RP}^2\setminus \mathbb{H}_{e_3}$. So, $\big(\mathbb{RP}^2\setminus \mathbb{H}_{e_3},~\oplus\big)$ forms a commutative group. Now, for all $(x_1:y_1:z_1),(x_2:y_2:z_2)\in\mathbb{RP}^2\setminus \mathbb{H}_{e_3}$ and for all $a,b\in \mathbb{R}$,
	\begin{equation*}
		ab\odot(x_1:y_1:z_1):=(abx_1:aby_1:z_1)=a\odot(bx_1:by_1:z_1)=a\odot\big( b\odot (x_1:y_1:z_1)\big),
	\end{equation*}
	\begin{align*}
		a\odot(x_1:y_1:z_1)\oplus a\odot(x_2:y_2:z_2)=&(ax_1:ay_1:z_1)\oplus (ax_2:ay_2:z_2)\\
		=&\big(a(x_1z_2+x_2z_1):a(y_1z_2+y_2z_1):z_1z_2\big)\\
		=&a\odot \big((x_1:y_1:z_1)\oplus(x_2:y_2:z_2)\big)
	\end{align*}
	and
	\begin{align*}
		a\odot(x_1:y_1:z_1)\oplus b\odot(x_1:y_1:z_1)=&(ax_1:ay_1:z_1)\oplus (bx_1:by_1:z_1)\\
		=&\Big((a+b)x_1z_1:(a+b)y_1z_1:z_1^2\Big)\\
		=&(a+b)\odot(x_1:y_1:z_1).
	\end{align*}
	Hence $\big(\mathbb{RP}^2\setminus \mathbb{H}_{e_3},~\oplus,~\odot\big)$ forms a vector space over $\mathbb{R}$.		
\end{proof}
\begin{remark}
	Note that in $\mathbb{RP}^2\setminus \mathbb{H}_{e_3}$ the $z$-axis (that is the line $x=0$, $y=0$) is the zero element. Simply, we denote it by $(0:0:\lambda)$, $\lambda\neq 0$.
\end{remark}
We use the notation $\ominus$ to indicate the difference between the two elements in $\mathbb{RP}^2\setminus \mathbb{H}_{e_3}$. That is if $(x_1:y_1:z_1),(x_2:y_2:z_2)\in\mathbb{RP}^2\setminus \mathbb{H}_{e_3}$, then $(x_1:y_1:z_1)\ominus(x_2:y_2:z_2)=(x_1z_2-x_2z_1:y_1z_2-y_2z_1:z_1z_2)$. So, each element $(x:y:z)$ in $\mathbb{RP}^2\setminus \mathbb{H}_{e_3}$ can be expressed as a sum of two of its elements  namely, $(x:0:z)$ and $(0:y:z)$. That is $(x:y:z)=(x:0:z)\oplus(0:y:z)$. Let $\mathbb{H}_{10}:=\left\{(x:0:z)\in\mathbb{RP}^2\setminus \mathbb{H}_{e_3}\right\}$ and $\mathbb{H}_{01}:=\left\{(0:y:z)\in\mathbb{RP}^2\setminus \mathbb{H}_{e_3}\right\}$. Then  $\mathbb{RP}^2\setminus \mathbb{H}_{e_3}$ can be expressed as
\begin{equation}\label{decoofrp2}
	\mathbb{RP}^2\setminus \mathbb{H}_{e_3}=\mathbb{H}_{10}\oplus \mathbb{H}_{01}.
\end{equation}
For the existence of an attractor of a contractive  RPIFS, we need to define a norm on $\mathbb{RP}^2\setminus \mathbb{H}_{e_3}$ for which the space $\mathbb{RP}^2\setminus \mathbb{H}_{e_3}$ becomes a complete normed linear space. For this purpose we define the real projective norm on $\mathbb{RP}^2\setminus \mathbb{H}_{e_3}$  as follows:
\begin{equation}
	\Vert(x:y:z)\Vert_{\mathbb{P}}:=\frac{\sqrt{x^2+y^2}}{\vert z\vert}
\end{equation}
for all $(x:y:z)\in\mathbb{RP}^2\setminus \mathbb{H}_{e_3}$. Since for $\lambda\neq 0$,
\begin{align*}
	\Vert(\lambda x:\lambda y:\lambda z)\Vert_{\mathbb{P}}=\frac{\sqrt{(\lambda x)^2+(\lambda y)^2}}{\vert\lambda z\vert}
	=\frac{\sqrt{x^2+y^2}}{\vert z\vert}
	=\Vert(x:y:z)\Vert_{\mathbb{P}}.
\end{align*}
So, $	\Vert \cdot\Vert_{\mathbb{P}}$ is well defined on  $\mathbb{RP}^2\setminus \mathbb{H}_{e_3}$.
We define the real projective metric $d_{\mathbb{P}}$ on $\mathbb{RP}^2\setminus \mathbb{H}_{e_3}$ as
\begin{equation}\label{pms}
	d_{\mathbb{P}}\big((x:y:z),(x':y':z')\big):=\Vert(xz'-x'z:yz'-y'z:zz')\Vert_{\mathbb{P}}
\end{equation}
for all $(x:y:z),(x':y':z')\in\mathbb{RP}^2\setminus \mathbb{H}_{e_3}$. Then
$\big(\mathbb{RP}^2\setminus \mathbb{H}_{e_3},d_{\mathbb{P}}\big)$ forms a metric space. It is clear that the projective metric $d_{\mathbb{P}}$ is neither equal to the ``round" metric  nor equal to the ``Hilbert" metric.
\begin{theorem}
	The metric space $\big(\mathbb{RP}^2\setminus \mathbb{H}_{e_3},d_{\mathbb{P}}\big)$ is complete.
\end{theorem}
\begin{proof}
	Let $(x_n:y_n:z_n)$ be a Cauchy sequnce in $\mathbb{RP}^2\setminus \mathbb{H}_{e_3}$ and let $\epsilon>0$. Then there exists a natural number $K$ such that
	\begin{align*}
		d_{\mathbb{P}}\big((x_n:y_n:z_n),(x_m:y_m:z_m)\big)<\epsilon\quad\mbox{for all}~n,m>K.
	\end{align*}
	This implies
	\begin{align*}
		\sqrt{\left(\frac{x_n}{z_n}-\frac{x_m}{z_m}\right)^2+\left(\frac{y_n}{z_n}-\frac{y_m}{z_m}\right)^2}<\epsilon\quad\mbox{for all}~n,m>K.
	\end{align*}
	This shows that $(\frac{x_n}{z_n})$ and $(\frac{y_n}{z_n})$ are Cauchy sequences in $\mathbb{R}$. So, there exist $x$ and $y$ in $\mathbb{R}$ such that $\frac{x_n}{z_n}\to x$ and $\frac{y_n}{z_n}\to y$. Now, for $\lambda\neq 0$
	\begin{align*}
		d_{\mathbb{P}}\big((x_n:y_n:z_n),(\lambda x:\lambda y:\lambda)\big)=\sqrt{\left(\frac{x_n}{z_n}-x\right)^2+\left(\frac{y_n}{z_n}-y\right)^2}.
	\end{align*}
	\\This shows that the sequence $(x_n:y_n:z_n)$  coverges to $(\lambda x:\lambda y:\lambda)$ on $\big(\mathbb{RP}^2\setminus \mathbb{H}_{e_3},d_{\mathbb{P}}\big)$. Hence $\big(\mathbb{RP}^2\setminus \mathbb{H}_{e_3},d_{\mathbb{P}}\big)$ is a complete metric space.
\end{proof}	
Before going to the  further discussions, we introduce some notation. For $(x_1:0:z_1),(x_2:0:z_2)\in \mathbb{H}_{10}$, we say that  $(x_1:0:z_1)\preceq (x_2:0:z_2)$, if and only if $x_1z_2\leq x_2z_1$, and $(x_1:0:z_1)\prec (x_2:0:z_2)$, if and only if $x_1z_2<x_2z_1$. Similarly for $(0:y_1:z_1),(0:y_2:z_2)\in \mathbb{H}_{01}$, we define $(0:y_1:z_1)\preceq(0:y_2:z_2)$, if and only if  $y_1z_2\leq y_2z_1$, and $(0:y_1:z_1)\prec(0:y_2:z_2)$, if and only if  $y_1z_2<y_2z_1$. Also, the product of two elements $(x_1:y_1:z_1)$ and $(x_2:y_2:z_2)$ in $\mathbb{RP}^2\setminus \mathbb{H}_{e_3}$ is defined by $(x_1:y_1:z_1)(x_2:y_2:z_2):=(x_1x_2:y_1y_2:z_1z_2)$.
\begin{definition}[Projective intervals on $\mathbb{H}_{10}$ and $\mathbb{H}_{01}$]
	Let $(a_1:0:c_1),(a_2:0:c_2)\in \mathbb{H}_{10}$ be such that $(a_1:0:c_1)\prec(a_2:0:c_2)$. Then the projective interval (see Figure~\ref{fig3}) on $\mathbb{H}_{10}$, is denoted  by $\mathbb{P}_{I\times\{0\}}$, and is defined by
	\begin{align*}
		\mathbb{P}_{I\times\{0\}}:=\bigg\{(x:0:z)\in\mathbb{H}_{10}:~(a_1:0:c_1)\preceq(x:0:z)\preceq(a_2:0:c_2) \bigg\}.
	\end{align*}
	One can define the projective interval on $\mathbb{H}_{01}$ in similar fashion.
\end{definition}
\begin{definition}[Projective rectangle]
	Let $(a_1:0:c_1),(a_2:0:c_2)\in \mathbb{H}_{10}$ and $(0:b_1:d_1),(0:b_2:d_2)\in \mathbb{H}_{01}$ be such that $(a_1:0:c_1)\prec(a_2:0:c_2)$ and $(0:b_1:d_1)\prec(0:b_2:d_2)$. Then the projective rectangle (see Figure~\ref{fig2}) on $\mathbb{RP}^2\setminus \mathbb{H}_{e_3}$ is defined by
	\begin{align*}
		\mathbb{P}_{I\times J}:=\bigg\{(x:y:z)\in\mathbb{RP}^2\setminus \mathbb{H}_{e_3}&:(a_1:0:c_1)\preceq(x:0:z)\preceq(a_2:0:c_2)\bigg.\\
		&\bigg. \mbox{and}~(0:b_1:d_1)\preceq(0:y:z)\preceq(0:b_2:d_2) \bigg\}.
	\end{align*}
\end{definition}
\begin{figure}[ht!]
	\centering
	\includegraphics[width=6cm, height=3cm]{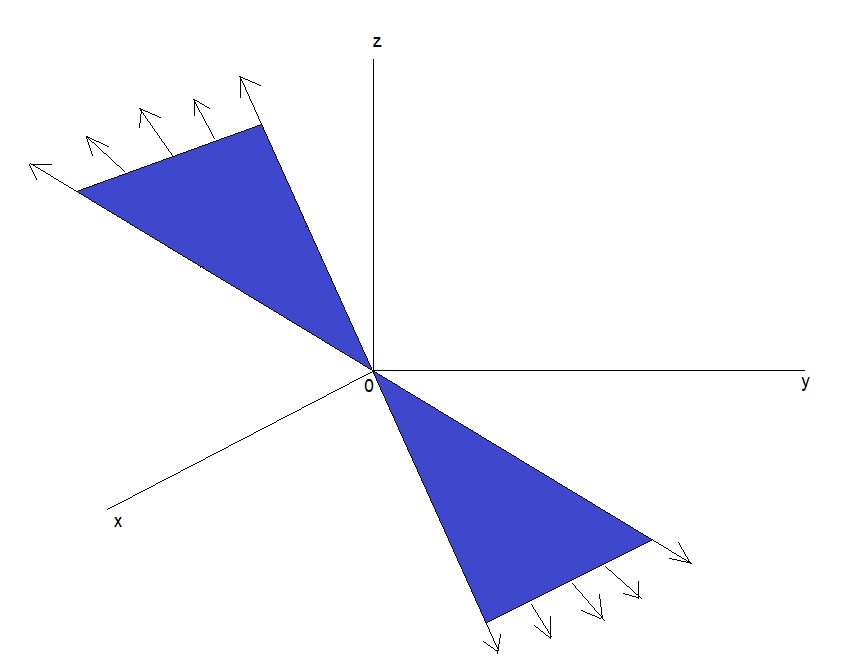}
	\caption{Projective interval.} \label{fig3}
\end{figure}
\begin{figure}[h!]
	\centering
	\includegraphics[width=5cm, height=5cm]{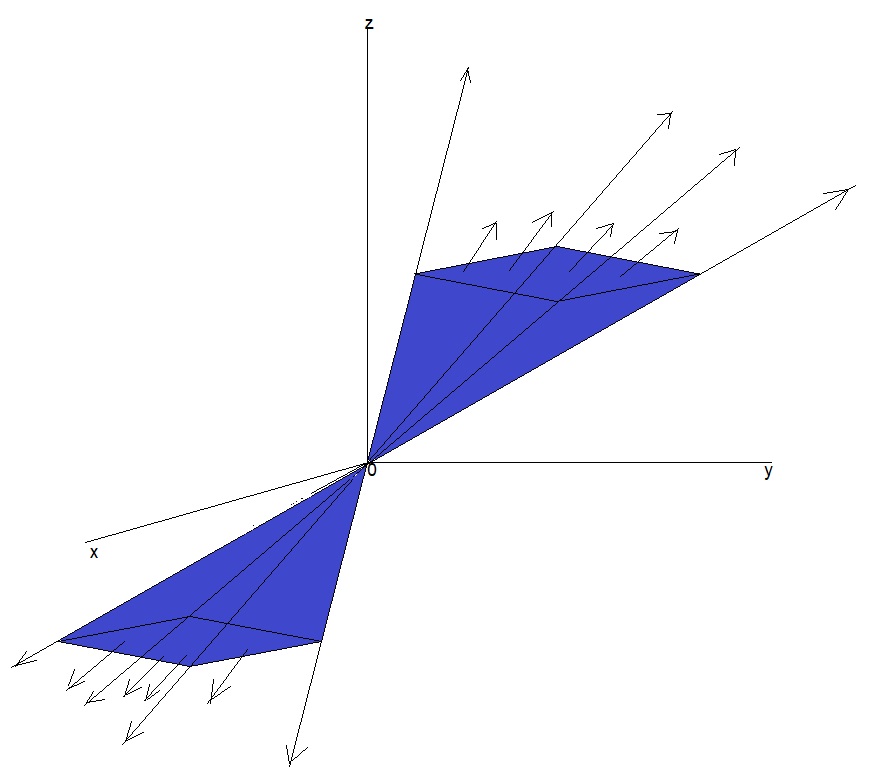}
	\caption{Projective rectangle.} \label{fig2}
\end{figure}

\begin{lemma}
	Projective intervals and projective rectangles are compact subsets of $\mathbb{RP}^2\setminus \mathbb{H}_{e_3}$ with respect to the metric $d_{\mathbb{P}}$.
\end{lemma}
\begin{proof}
	The proof follows from the definitions of the projective interval and the projective rectangle respectively.
\end{proof}
Let
\begin{equation}
	\mathscr{C}[\mathbb{P}_{I\times\{0\}}]=\bigg\{f: f:\mathbb{P}_{I\times\{0\}}\to\mathbb{H}_{01}~\mbox{is continuous}\bigg\}.
\end{equation}
If $f\in \mathscr{C}[\mathbb{P}_{I\times\{0\}}]$, define $\Vert f \Vert_{\mathbb{P}\infty}:=\sup\{\Vert f(x:0:z)\Vert_{\mathbb{P}}:~(x:0:z)\in \mathbb{P}_{I\times\{0\}}\}$. Since $\mathbb{P}_{I\times\{0\}}$ is compact so, $\Vert f \Vert_{\mathbb{P}\infty}$ is well defined.
\begin{remark}
	The space $\mathscr{C}[\mathbb{P}_{I\times\{0\}}]$ forms a normed linear space, where the addition is defined by $(f\oplus g)(x:0:z)=f(x:0:z)\oplus g(x:0:z)$ and the multiplication is defined by $(\alpha\odot f)(x:0:z)=\alpha\odot f(x:0:z)$.
\end{remark}
\begin{lemma}\label{comnessc01}
	$\big(\mathscr{C}[\mathbb{P}_{I\times\{0\}}],\Vert \cdot \Vert_{\mathbb{P}\infty}\big)$ is a complete normed linear space.
\end{lemma}
\begin{proof}
	Let $(f_n)$ be a Cauchy sequence in $\mathscr{C}[\mathbb{P}_{I\times\{0\}}]$ and let $\epsilon>0$. Then there exists a natural number $k_0$ such that
	$$\lVert f_n\ominus f_m\lVert_{\mathbb{P}\infty}<\epsilon$$ for $n,m>k_0$. Then for each $(x:0:z)\in \mathbb{P}_{I\times\{0\}}$,
	\begin{equation}\label{cauchyforex}
		\Vert f_n(x:0:z)\ominus f_m(x:0:z)\Vert_{\mathbb{P}}<\epsilon
	\end{equation}
	for $n,m>k_0$. Therefore, $\big(f_n(x:0:z)\big)$ is Cauchy in $\mathbb{H}_{01}$. As $\mathbb{H}_{01}$ is closed in  $\big(\mathbb{RP}^2\setminus \mathbb{H}_{e_3},d_{\mathbb{P}}\big)$, $\mathbb{H}_{01}$ is complete. So, $\big(f_n(x:0:z)\big)$ converges to a point $(0:v:w)$. Define a function $f$ on $\mathbb{P}_{I\times\{0\}}$ by $f(x:0:z)=(0:v:w)$. Now if $m$ is large enough, then from (\ref{cauchyforex}),
	\begin{equation*}
		\Vert f(x:0:z)\ominus f_m(x:0:z)\Vert_{\mathbb{P}}=\lim\limits_{n\to \infty} \Vert f_n(x:0:z)\ominus f_m(x:0:z)\Vert_{\mathbb{P}}<\epsilon.
	\end{equation*}
	This is true for each $(x:0:z)\in \mathbb{P}_{I\times\{0\}}$. So,
	\begin{equation*}
		\sup_{(x:0:z)\in \mathbb{P}_{I\times\{0\}}}\Vert f(x:0:z)\ominus f_m(x:0:z)\Vert_{\mathbb{P}}\leq\epsilon.
	\end{equation*}
	Therefore, $$\lVert f\ominus f_m\lVert_{\mathbb{P}\infty}\leq\epsilon\quad\mbox{as}~m\to\infty.$$ The continuity of $f$  follows from the continuity of $f_n$.  Therefore, $f\in \mathscr{C}[\mathbb{P}_{I\times\{0\}}]$. Hence  $\big(\mathscr{C}[\mathbb{P}_{I\times\{0\}}],\Vert \cdot \Vert_{\mathbb{P}\infty}\big)$ is complete.
\end{proof}

\section{Real projective fractal interpolation function}\label{rpfifdef}
In this section, we construct the real projective fractal interpolation function  passing through certain data points on $\mathbb{RP}^2\setminus \mathbb{H}_{e_3}$. \\
Let $N\geq 2$ and $\left\{(x_n:y_n:z_n)\in\mathbb{RP}^2\setminus \mathbb{H}_{e_3}:~n=0,1,\ldots, N\right\}$ be a data set in $\mathbb{RP}^2\setminus \mathbb{H}_{e_3}$ such that $x_nz_{n+1}<x_{n+1}z_n$ for $n=0,1,\ldots, N-1$. Let $\mathbb{P}_{I\times\{0\}}:=\bigg\{(x:0:z)\in\mathbb{H}_{10}:~(x_0:0:z_0)\preceq(x:0:z)\preceq(x_N:0:z_N) \bigg\}$ and $\mathbb{P}_{I_{n}\times\{0\}}:=\bigg\{(x:0:z)\in\mathbb{H}_{10}:~(x_{n-1}:0:z_{n-1})\preceq(x:0:z)\preceq(x_n:0:z_n) \bigg\}$ for $n=1,2,\ldots, N$. For $n=1,2,\ldots, N$, consider the transformations $L_n:\mathbb{P}_{\mathbb{I}\times\{0\}}\to \mathbb{P}_{\mathbb{I}_n\times\{0\}}$ given by $L_n(x:0:z)=(a_nx+b_nz:0:z)$ such that 
\begin{equation}\label{joufl_n}
	L_n(x_0:0:z_0)=(x_{n-1}:0:z_{n-1})~\mbox{and}~L_n(x_N:0:z_N)=(x_{n}:0:z_{n}),
\end{equation}
where $a_n,b_n\in\mathbb{R}$. The constants $ a_n$ and $b_n$ are determined
by the condition (\ref{joufl_n}) as
\begin{align*}
	a_n=\frac{\frac{x_n}{z_n}-\frac{x_{n-1}}{z_{n-1}}}{\frac{x_N}{z_N}-\frac{x_0}{z_0}}\quad\mbox{and}\quad b_n=\frac{\frac{x_N}{z_N}\frac{x_{n-1}}{z_{n-1}}-\frac{x_{0}}{z_{0}}\frac{x_{n}}{z_{n}}}{\frac{x_N}{z_N}-\frac{x_0}{z_0}}.
\end{align*}
It is clear that $\vert a_n\vert<1$. Also,	 
\begin{align}\label{conolnms}
	d_{\mathbb{P}}\big(L_n(x:0:z),L_n(x':0:z')\big)=&d_{\mathbb{P}}\big((a_nx+b_nz:0:z),(a_nx'+b_nz':0:z')\big)\\
	\nonumber	=&\frac{\sqrt{\big((a_nx+b_nz)z'-(a_nx'+b_nz')z\big)^2}}{\vert zz'\vert}\\
	\nonumber	=&\vert a_n\vert\frac{\vert xz'-x'z\vert}{\vert zz'\vert}=\vert a_n\vert d_{\mathbb{P}}\big((x:0:z),(x':0:z')\big).
\end{align}
So, $L_n$'s are contraction maps. Also, for $n=1,2,\ldots, N$, consider the continuous maps $F_n:\mathbb{RP}^2\setminus \mathbb{H}_{e_3}\to\mathbb{H}_{01}$ given by
\begin{equation}\label{fnitsjoin}
	F_n(x:y:z)=\big(0:c_nx+d_ny+f_nz:z\big)
\end{equation}
such that 
\begin{equation}\label{joucofn}
	F_n(x_0:y_0:z_0)=(0:y_{n-1}:z_{n-1})\quad\mbox{and}\quad F_n(x_N:y_N:z_N)=(0:y_{n}:z_{n}),
\end{equation}
where $c_n,d_n,f_n\in\mathbb{R}$. The real constants $c_n$ and $f_n$ are determined by the condition (\ref{joucofn}) as
\begin{align*}
	c_n=\frac{\frac{y_n}{z_n}-\frac{y_{n-1}}{z_{n-1}}}{\frac{x_N}{z_N}-\frac{x_0}{z_0}}-d_n\frac{\frac{y_N}{z_N}-\frac{y_{0}}{z_{0}}}{\frac{x_N}{z_N}-\frac{x_0}{z_0}}\quad\mbox{and}\\ f_n=\frac{\frac{x_N}{z_N}\frac{y_{n-1}}{z_{n-1}}-\frac{x_{0}}{z_{0}}\frac{y_{n}}{z_{n}}}{\frac{x_N}{z_N}-\frac{x_0}{z_0}}-d_n\frac{\frac{x_N}{z_N}\frac{y_{0}}{z_{0}}-\frac{x_{0}}{z_{0}}\frac{y_{N}}{z_{N}}}{\frac{x_N}{z_N}-\frac{x_0}{z_0}}.
\end{align*}
Here, $d_n$'s are the free parameters.
Also, we get the following.
\begin{align}\label{conoffmswrfc}
	d_\mathbb{P}&\left(F_n\big((x:0:z)\oplus(0:y:z)\big),F_n\big((x':0:z')\oplus(0:y:z)\big)\right)\\
	\nonumber	     &=d_\mathbb{P}\left(F_n\big((x:y:z),F_n(x'z:yz':zz')\big)\right)\\
	\nonumber		&=d_\mathbb{P}\left(\big(0:c_nx+d_ny+f_nz:z\big),\big(0:c_nx'z+d_nyz'+f_nzz':zz'\big)\right)\\
	\nonumber&=\frac{\sqrt{\big((c_nx+d_ny+f_nz)zz'-(c_nx'z+d_nyz'+f_nzz')z\big)^2}}{\vert z^2z'\vert}\\
	\nonumber		&=\frac{\sqrt{\big(c_nxzz'-c_nx'z^2\big)^2}}{\vert z^2z'\vert}\\
	\nonumber	&=\vert c_n\vert\frac{\vert xz'-x'z\vert}{\vert zz'\vert}=\vert c_n\vert d_{\mathbb{P}}\big((x:0:z),(x':0:z')\big).
\end{align}
Similarly, we have 
\begin{align}\label{conffnswrscs}
	d_\mathbb{P}&\left(F_n\big((x:0:z)\oplus(0:y:z)\big),F_n\big((x:0:z)\oplus(0:y':z')\big)\right)\\
	&=\vert d_n\vert d_{\mathbb{P}}\big((0:y:z),(0:y':z')\big).
\end{align}
This shows that $F_n$'s are Lipschitz.
Now, for $n=1,2,\ldots,N$, define the functions $W_n:\mathbb{RP}^2\setminus \mathbb{H}_{e_3}\to\mathbb{RP}^2\setminus \mathbb{H}_{e_3}$ by 
\begin{equation}\label{protfs}
	W_n(x:y:z)=L_n(x:0:z)\oplus F_n(x:y:z).
\end{equation}
Then the maps $W_n$ can also be expressed as 
\begin{align}\label{expofrpifs}
	W_n(x:y:z)=&(a_nx+b_nz:0:z)\oplus (0:c_nx+d_ny+f_nz:z)\\
	\nonumber	=&(a_nx+b_nz:c_nx+d_ny+f_nz:z)\\
	\nonumber		=&(a_nx:c_nx+d_ny:z)\oplus (b_n:f_n:1)\\
	\nonumber		=&\begin{pmatrix}
		a_n & 0 & 0\\
		c_n & d_n & 0\\
		0 & 0 & 1
	\end{pmatrix} \begin{pmatrix}
		x\\ y\\ z
	\end{pmatrix}\oplus\begin{pmatrix}
		b_n\\f_n\\1
	\end{pmatrix}\\
	\nonumber	=&\begin{pmatrix}
		a_n & 0 & b_n\\
		c_n & d_n & f_n\\
		0 & 0 & 1
	\end{pmatrix} \begin{pmatrix}
		x\\ y\\ z
	\end{pmatrix},
\end{align}
where $ \begin{pmatrix}
	x\\ y\\ z
\end{pmatrix} $ represents the element $(x: y: z)$ in $\mathbb{RP}^2\setminus \mathbb{H}_{e_3}$. For non-zero $d_n$'s, $W_n$'s are non-singular transformations. Then $\big\{\mathbb{RP}^2\setminus \mathbb{H}_{e_3}; W_n:~ n=1,2,\ldots,N\big\}$
forms a  RPIFS. Note that $W_n$'s satisfy the join up conditions 
$W_n(x_0:y_0:z_0)=(x_{n-1}:0:z_{n-1})\oplus(0:y_{n-1}:z_{n-1})=(x_{n-1}:y_{n-1}:z_{n-1})$  and  $W_n(x_N:y_N:z_N)=(x_{n}:0:z_{n})\oplus(0:y_{n}:z_{n})=(x_{n}:y_{n}:z_{n})$. 
It can be seen that the projective transformation $W_n$ defined in (\ref{protfs}) maps the line segment $L$ (in $\mathbb{RP}^2\setminus \mathbb{H}_{e_3}$) parallel to the line $x=0$ into the line segment $W_n(L)$ (in $\mathbb{RP}^2\setminus \mathbb{H}_{e_3}$) parallel to the line $x=0$ so that the  ratio of the length of $W_n(L)$ to the length of $L$ is $\vert d_n\vert$. The maps $W_n$'s may or may not be contractive with respect to the real projective metric $d_{\mathbb{P}}$. But if $W_n$'s are contractive, then  Figure~\ref{prftfns} illustrates that $W_n$ maps a projective rectangle to a projective rectangle. 
\begin{figure}[h!]
	\centering
	\includegraphics[width=7cm, height=4cm]{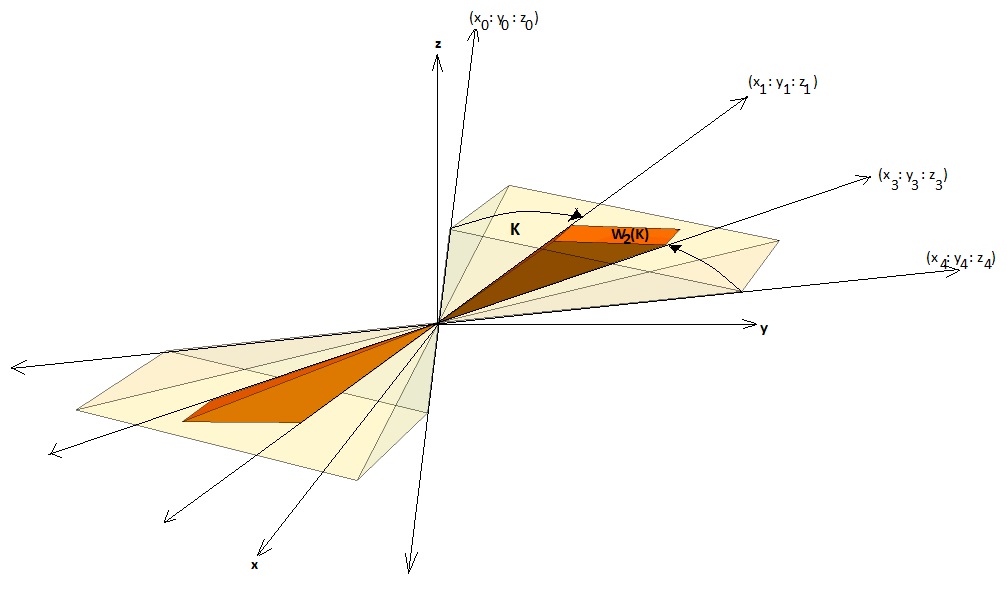}
	\caption{$W_2$ transforms the projective rectangle to a smaller projective rectangle.} \label{prftfns}
\end{figure}

Let $\theta$ be a positive real number. We define a new metric $d_{\theta}$ on $\mathbb{RP}^2\setminus \mathbb{H}_{e_3}$ as follows
\begin{equation*}
	d_{\theta}\big((x:y:z),(x':y':z')\big):=d_{\mathbb{P}}\big((x:0:z),(x':0:z')\big)+\theta d_{\mathbb{P}}\big((0:y:z),(0:y':z')\big).
\end{equation*}
\begin{lemma}
	The metric $d_{\theta}$ is equivalent to the metric $d_{\mathbb{P}}$.
\end{lemma}
\begin{proof}
	For $a>0,b>0$, $\vert a\vert\leq \sqrt{a^2+b^2}$. So,
	\begin{align*}
		d_{\mathbb{P}}\big((x:0:z),(x':0:z')\big)+&\theta d_{\mathbb{P}}\big((0:y:z),(0:y':z')\big)\\
		=&\frac{\vert xz'-x'z\vert}{\vert zz'\vert}+\theta\frac{\vert yz'-y'z\vert}{\vert zz'\vert}\\
		\leq &(1+\theta)\frac{\sqrt{(xz'-x'z)^2+(yz'-y'z)^2}}{\vert zz'\vert}\\
		=&(1+\theta)d_{\mathbb{P}}\big((x:y:z),(x':y':z')\big).
	\end{align*}
	Therefore,
	\begin{equation*}
		d_{\theta}\big((x:y:z),(x':y':z')\big)\leq (1+\theta)d_{\mathbb{P}}\big((x:y:z),(x':y':z')\big).
	\end{equation*}
	{\bf Case 1.} If $\theta\geq 1$, then $d_{\mathbb{P}}\big((x:0:z),(x':0:z')\big)+\theta d_{\mathbb{P}}\big((0:y:z),(0:y':z')\big)\geq \frac{\vert  xz'-x'z\vert }{\vert zz'\vert }+\frac{\vert yz'-y'z\vert }{\vert zz'\vert }$. If $a>0,b>0$, then $\vert a+b\vert \geq\sqrt{a^2+b^2}$. Therefore,
	\begin{align*}
		d_{\theta}\big((x:y:z),(x':y':z')\big)&\geq\frac{\sqrt{(xz'-x'z)^2+(yz'-y'z)^2}}{\vert zz'\vert }\\
		&=d_{\mathbb{P}}\big((x:y:z),(x':y':z')\big).
	\end{align*}
	Hence $d_{\mathbb{P}}\big((x:y:z),(x':y':z')\big)\leq d_{\theta}\big((x:y:z),(x':y':z')\big)
	\leq (1+\theta)d_{\mathbb{P}}\big((x:y:z),(x':y':z')\big).$\\
	{\bf Case 2.} If $\theta<1$, then $\frac{1}{\theta}>1$, So, $\frac{1}{\theta}d_{\theta}\big((x:y:z),(x':y':z')\big)=\frac{1}{\theta}d_{\mathbb{P}}\big((x:):z),(x':0:z')\big)+d_{\mathbb{P}}\big((0:y:z),(0:y':z')\big).$ Then by similar arguments as in {\bf Case 1}, we have
	$\theta d_{\mathbb{P}}\big((x:y:z),(x':y':z')\big)\leq d_{\theta}\big((x:y:z),(x':y':z')\big)\leq (1+\theta)d_{\mathbb{P}}\big((x:y:z),(x':y':z')\big).$
	Therefore, the metric $d_{\theta}$ is equivalent to the metric $d_{\mathbb{P}}$.
\end{proof}

\begin{theorem}
	If $0<\theta\leq\frac{\min\{1-2\vert c_n\vert:~n=1,2,\ldots,N\}}{\max\{2\vert a_n\vert:~n=1,2,\ldots,N\}}$, $a=\max\{\vert a_n\vert+\theta \vert c_n\vert:~n=1,2,\ldots,N\}$,  $d=\max\{\vert d_n\vert:~n=1,2,\ldots,N\}<1$ and $c=\max\{a,d\}$, then the maps $W_n$'s are contractive with respect to the metric $d_{\theta}$ and the contraction factor $c$.
\end{theorem}
\begin{proof}
	Since $W_n(x:y:z)=(a_nx+b_nz:c_nx+d_ny+f_nz:z)$, for $(x:y:z),(x':y':z')\in \mathbb{RP}^2\setminus \mathbb{H}_{e_3}$,
	\begin{align*}
		d_{\theta}&\big(W_n(x:y:z),W_n(x':y':z')\big)\\
		=d_{\theta}&\big((a_nx+b_nz:c_nx+d_ny+f_nz:z),(a_nx'+b_nz':c_nx'+d_ny'+f_nz':z')\big)\\
		=d_{\mathbb{P}}&\big((a_nx+b_nz:0:z),(a_nx'+b_nz':0:z')\big)\\
		&+\theta d_{\mathbb{P}}\big((0:c_nx+d_ny+f_nz:z),(0:c_nx'+d_ny'+f_nz':z')\big)\\
		=d_{\mathbb{P}}&\big(L_n(x:0:z),L_n(x':0:z')\big)+\theta d_{\mathbb{P}}\big(F_n(x:y:z),F_n(x':y':z')\big)\\
		\leq d_{\mathbb{P}}&\big(L_n(x:0:z),L_n(x':0:z')\big)+\theta d_{\mathbb{P}}\big(F_n(x:y:z),F_n(xz':y'z:zz')\big)\\
		&+\theta d_{\mathbb{P}}\big(F_n(xz':y'z:zz'),F_n(x':y':z')\big)\quad\mbox{(by triangular inequality)}.
	\end{align*}
	Using (\ref{conolnms}), (\ref{conoffmswrfc}) and (\ref{conffnswrscs}), we get
	\begin{align*}
		d_{\theta}&\big(W_n(x:y:z),W_n(x':y':z')\big)\\ 
		\leq &\vert a_n\vert d_{\mathbb{P}}\big((x:0:z),(x':0:z')\big)+\theta \vert d_n\vert d_{\mathbb{P}}\big((0:y:z),(0:y':z')\big)\\
		&+\theta \vert c_n\vert d_{\mathbb{P}}\big((x:0:z),(x':0:z')\big)\\
		=&\big(\vert a_n\vert+\theta\vert c_n\vert\big)d_{\mathbb{P}}\big((x:0:z),(x':0:z')\big)+\theta \vert d_n\vert d_{\mathbb{P}}\big((0:y:z),(0:y':z')\big)\\
		\leq& a~d_{\mathbb{P}}\big((x:0:z),(x':0:z')\big)+\theta d~d_{\mathbb{P}}\big((0:y:z),(0:y':z')\big)\\
		\leq& c~\bigg(d_{\mathbb{P}}\big((x:0:z),(x':0:z')\big)+\theta d_{\mathbb{P}}\big((0:y:z),(0:y':z')\big)\bigg)\\
		=&c~d_{\theta}\big((x:y:z),(x':y':z')\big).
	\end{align*}
	Since $\theta\leq \frac{1-2\vert c_n\vert }{2\vert a_n\vert }$, this implies $\big(\vert a_n\vert +\theta\vert c_n\vert \big)\leq\frac{1}{2}<1$ for $n=1,2,\ldots,N$. Therefore, $a<1$. Also, if we consider $\vert d_n\vert <1$, then $d<1$. This shows that $c<1$. Hence $W_n$'s are contraction maps. 
\end{proof}
Since the space $\big(\mathbb{RP}^2\setminus \mathbb{H}_{e_3},d_{\mathbb{P}}\big)$ is complete. So, the RPIFS $\big\{\mathbb{RP}^2\setminus \mathbb{H}_{e_3}; W_n:~ n=1,2,\ldots,N\big\}$ has an unique attractor, say $G$.  Now, we show that $G$ is the graph of a continuous function from $\mathbb{P}_{\mathbb{I}\times\{0\}}$ to $\mathbb{H}_{01}$. An illustration is provided in Figure~\ref{insidevrfif}, for $N=3$.
\begin{figure}[h!]
	\centering
	\includegraphics[width=5cm, height=5cm]{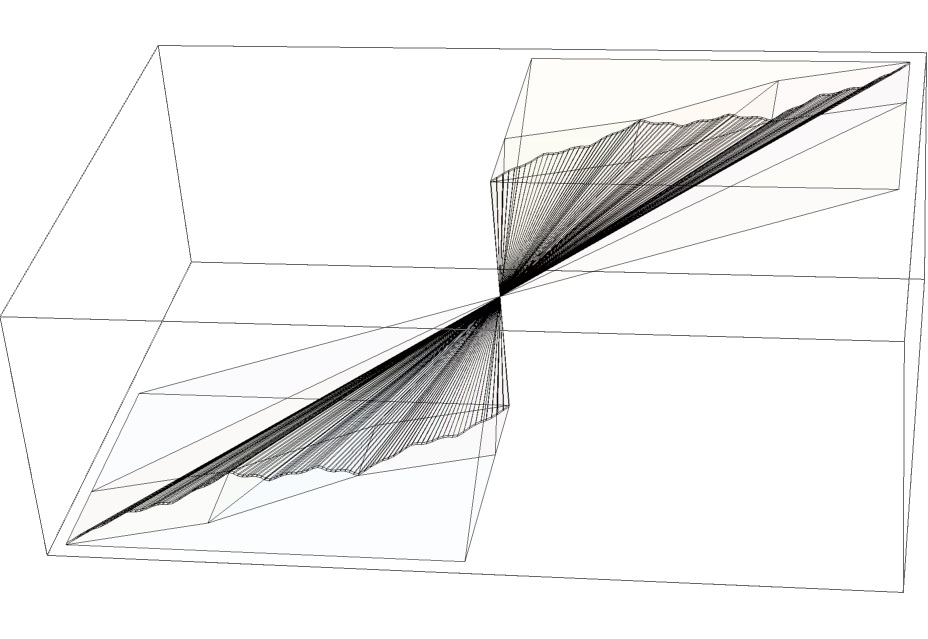}
	\caption{Inside view of construction of the graph of a RPFIF.} \label{insidevrfif}
\end{figure}

Over the projective interval $\mathbb{P}_{\mathbb{I}\times\{0\}}$, we consider the space of continuous functions $$\mathscr{C}[\mathbb{P}_{\mathbb{I}\times\{0\}}]:=\left\{f:~f:\mathbb{P}_{\mathbb{I}\times\{0\}}\to\mathbb{H}_{01}~\mbox{is continuous}\right\}.$$ 
Then from Lemma \ref{comnessc01}, the space $\mathscr{C}[\mathbb{P}_{\mathbb{I}\times\{0\}}]$ is complete with respect to $\Vert \cdot \Vert_{\mathbb{P}\infty}$. Let $$\mathscr{F}:=\left\{f\in\mathscr{C}[\mathbb{P}_{\mathbb{I}\times\{0\}}]:~f(x_0:0:z_0)=(0:y_0:z_0),~f(x_N:0:z_N)=(0:y_N:z_N) \right\}.$$ Then $\mathscr{F}$ is a closed subset of $\big(\mathscr{C}[\mathbb{P}_{\mathbb{I}\times\{0\}}],\Vert \cdot \Vert_{\mathbb{P}\infty}\big)$. So, $\mathscr{F}$ is complete. Finally, we define a  {\bf real projective Read-Bajraktarevic}-operator (RPRB)  $\mathcal{T}:\mathscr{F}\to \mathscr{F}$,  as follows.
\begin{equation}
	(\mathcal{T}f)(x:0:z):=F_n\big(L_n^{-1}(x:0:z)\oplus f\circ L_n^{-1}(x:0:z)\big)
\end{equation}
whenever $(x:0:z)\in \mathbb{P}_{\mathbb{I}_n\times\{0\}}$ for $n=1,2,\ldots, N$. 
\begin{theorem}
	The RPRB-operator $\mathcal{T}$ is well defined on $\mathscr{F}$.
\end{theorem} 
\begin{proof}
	For $(x_0:0:z_0)\in \mathbb{P}_{\mathbb{I}\times\{0\}}$,
	\begin{align*}
		(\mathcal{T}f)(x_0:0:z_0)&=F_1\big(L_1^{-1}(x_0:0:z_0)\oplus f\circ L_1^{-1}(x_0:0:z_0)\big)\\
		&=F_1\big((x_0:0:z_0)\oplus f(x_0:0:z_0)\big)\\
		&=F_1\big((x_0:0:z_0)\oplus(0:y_0:z_0)\big)\\
		&=F_1\big((x_0:y_0:z_0)\big)=(0:y_0:z_0).
	\end{align*}
	Similarly, $(\mathcal{T}f)(x_N:0:z_N)=(0:y_N:z_N).$\\
	Also, whenever $(x_n:0:z_n)\in \mathbb{P}_{\mathbb{I}_n\times\{0\}}$, then
	\begin{align*}
		(\mathcal{T}f)(x_n:0:z_n)&=F_n\big(L_n^{-1}(x_n:0:z_n)\oplus f\circ L_n^{-1}(x_n:0:z_n)\big)\\
		&=F_n\big((x_N:0:z_N)\oplus f(x_N:0:z_N)\big)\\
		&=F_n\big((x_N:0:z_N)\oplus(0:y_N:z_N)\big)\\
		&=F_n\big((x_N:y_N:z_N)\big)=(0:y_n:z_n)
	\end{align*}
	and whenever $(x_n:0:z_n)\in \mathbb{P}_{\mathbb{I}_{n+1}\times\{0\}}$, then
	\begin{align*}
		(\mathcal{T}f)(x_n:0:z_n)&=F_{n+1}\big(L_{n+1}^{-1}(x_n:0:z_n)\oplus f\circ L_{n+1}^{-1}(x_n:0:z_n)\big)\\
		&=F_{n+1}\big((x_0:0:z_0)\oplus f(x_0:0:z_0)\big)\\
		&=F_{n+1}\big((x_0:0:z_0)\oplus(0:y_0:z_0)\big)\\
		&=F_{n+1}\big((x_0:y_0:z_0)\big)=(0:y_n:z_n).
	\end{align*}
	This shows that $\mathcal{T}f$ is well defined and $\mathcal{T}f\in \mathscr{F}$.
\end{proof}
\begin{theorem}
	The RPRB-operator $\mathcal{T}$ is contractive on $\big(\mathscr{F},\Vert \cdot \Vert_{\mathbb{P}\infty}\big)$.
\end{theorem}

\begin{proof}
	Let $f,g\in\mathscr{F}$. Then for $(x:0:z)\in \mathbb{P}_{\mathbb{I}_n\times\{0\}}$
	\begin{align}\label{tconsd}
		\Vert &(\mathcal{T}f)(x:0:z)\ominus (\mathcal{T}g)(x:0:z) \Vert_{\mathbb{P}}\\
		\nonumber=&\Vert F_n\big(L_n^{-1}(x:0:z)\oplus f\circ L_n^{-1}(x:0:z)\big)\ominus F_n\big(L_n^{-1}(x:0:z)\oplus g\circ L_n^{-1}(x:0:z)\big) \Vert_{\mathbb{P}}.
	\end{align}
	For simplicity,  $L_n^{-1}(x:0:z)=(x-b_nz:0:a_nz)=(x_1:0:z_1)$, (say). Then $f\circ L_n^{-1}(x:0:z)=f(x_1:0:z_1)=(0:y_2:z_2)$ and $g\circ L_n^{-1}(x:0:z)=g(x_1:0:z_1)=(0:y_3:z_3)$, (say).
	It follows that
	\begin{align}\label{expss}
		F_n\big(L_n^{-1}(x:0:z)\oplus f\circ L_n^{-1}(x:0:z)\big)=&F_n\big((x_1:0:z_1)\oplus (0:y_2:z_2)\big)\\
		\nonumber			=&F_n(x_1z_2:y_2z_1:z_1z_2)\\
		\nonumber			=&(0:c_nx_1z_2+d_ny_2z_1+f_nz_1z_2:z_1z_2)
	\end{align}
	and 
	\begin{align}\label{edplbbh}
		F_n\big(L_n^{-1}(x:0:z)\oplus g\circ L_n^{-1}(x:0:z)\big)=(0:c_nx_1z_3+d_ny_3z_1+f_nz_1z_3:z_1z_3).
	\end{align}
	Therefore, (\ref{tconsd}), (\ref{expss}) and (\ref{edplbbh}) together give
	\begin{align*}
		\Vert& (\mathcal{T}f)(x:0:z)\ominus (\mathcal{T}g)(x:0:z) \Vert_{\mathbb{P}}\\
		=&\lVert\big((0:c_nx_1z_2+d_ny_2z_1+f_nz_1z_2:z_1z_2)\ominus (0:c_nx_1z_3+d_ny_3z_1+f_nz_1z_3:z_1z_3)\big)\rVert_{\mathbb{P}}\\
		=&\lVert\big(0:(c_nx_1z_2+d_ny_2z_1+f_nz_1z_2)z_1z_3\\&-(c_nx_1z_3+d_ny_3z_1+f_nz_1z_3)z_1z_2:z^2_1z_2z_3)\big)\rVert_{\mathbb{P}}\\
		=&\lVert\big(0:d_n(y_2z_3-y_3z_2)z^2_1:z^2_1z_2z_3)\big)\rVert_{\mathbb{P}}\\
		=&\lVert d_n\odot\big(0:y_2z_3-y_3z_2:z_2z_3)\big)\rVert_{\mathbb{P}}\\
		=&\lvert d_n\rvert\lVert \big((0:y_2:z_2)\ominus (0:y_3:z_3)\big)\rVert_{\mathbb{P}}\\
		=&\lvert d_n\rvert\lVert \big(f\circ L_n^{-1}(x:0:z)\ominus g\circ L_n^{-1}(x:0:z)\big)\rVert_{\mathbb{P}}.
	\end{align*}
	Since $L_n^{-1}(x:0:z)\in \mathbb{P}_{\mathbb{I}\times\{0\}}$.
	Hence
	\begin{align*}
		\Vert (\mathcal{T}f)(x:0:z)\ominus (\mathcal{T}g)(x:0:z) \Vert_{\mathbb{P}}
		&\leq\lvert d_n\rvert\lVert f\ominus g\rVert_{\mathbb{P}\infty}\\
		&\leq d~\lVert f\ominus g\rVert_{\mathbb{P}\infty},
	\end{align*}
	where $d=\max\{\vert d_n\vert:~n=1,2,\ldots,N\}<1$. Taking supremum over all $(x:0:z)\in\mathbb{P}_{\mathbb{I}\times\{0\}}$, we get
	\begin{align*}
		\Vert\mathcal{T}f\ominus\mathcal{T}g \Vert_{\mathbb{P}\infty}\leq d~\Vert {f}\ominus{g}\Vert_{\mathbb{P}\infty}.
	\end{align*}
	Hence the  RPRB-operator $\mathcal{T}$ is contractive on $\mathscr{F}$.
\end{proof}
Since $\big(\mathscr{F},\Vert \cdot \Vert_{\mathbb{P}\infty}\big)$ is complete, by Banach fixed point theorem $\mathcal{T}$ has an unique fixed point {\bf f} in $\mathscr{F}$. We call {\bf f} as the  {\bf real projective fractal interpolation function} (RPFIF) corresponding to the RPIFS  $\big\{\mathbb{RP}^2\setminus \mathbb{H}_{e_3}; W_n:~ n=1,2,\ldots,N\big\}$.  This proves the existence of a RPFIF {\bf f} in {\bf Theorem~\ref{maint11}}.
\begin{theorem}\label{mainthofif}
	The graph of the function {\bf f} is the attractor of the RPIFS $\big\{\mathbb{RP}^2\setminus \mathbb{H}_{e_3}; W_n:~ n=1,2,\ldots,N\big\}$. That is 
	$	G=graph({\bf f})$.
\end{theorem}

\begin{proof}
	Note that from (\ref{decoofrp2}), the graph of any continuous function  $f: \mathbb{H}_{10}\rightarrow \mathbb{H}_{01}$  can be expressed as  $$graph(f)=\left\{(x:0:z)\oplus{ f}(x:0:z):~(x:0:z)\in \mathbb{H}_{10}\right\}.$$ 
	Now, let $\tilde{G}=graph({\bf f}):=\left\{(x:0:z)\oplus{\bf f}(x:0:z):~(x:0:z)\in \mathbb{P}_{\mathbb{I}\times\{0\}}\right\}.$ Then 
	\begin{equation}\label{ghjdyhd}
		\bigcup_{n=1}^NW_n(\tilde{G})=\bigcup_{n=1}^N\left\{W_n\big((x:0:z)\oplus{\bf f}(x:0:z)\big):~(x:0:z)\in \mathbb{P}_{\mathbb{I}\times\{0\}}\right\}.
	\end{equation}
	Since ${\bf f}:\mathbb{P}_{\mathbb{I}\times\{0\}}\rightarrow\mathbb{H}_{01}$, ${\bf f}(x:0:z)=(0:v:w)$, (say). Then, we get 
	\begin{align}\label{hngker}
		W_n\big((x:0:z)\oplus{\bf f}(x:0:z)\big)&=W_n\big((x:0:z)\oplus(0:v:w)\big)\\
		\nonumber	&=W_n(xw:vz:zw)\\
		\nonumber	&=L_n(xw:0:zw)\oplus F_n(xw:vz:zw)\\
		\nonumber   &=L_n(x:0:z)\oplus F_n(xw:vz:zw).
	\end{align}
	Since, {\bf f} is the fixed point of the RPRB-operator $\mathcal{T}$. Therefore,
	\begin{align}\label{ghfuo}
		{\bf f}\big(L_n(x:0:z)\big)&=(\mathcal{T}{\bf f})\big(L_n(x:0:z)\big)\\
		\nonumber	&=F_n\big((x:0:z)\oplus{\bf f}(x:0:z)\big)\\
		\nonumber	&=F_n\big((x:0:z)\oplus(0:v:w)\big)\\
		\nonumber   &= F_n(xw:vz:zw). 
	\end{align}
	From (\ref{hngker}) and (\ref{ghfuo}), we get
	\begin{equation}\label{lwesf}
		W_n\big((x:0:z)\oplus{\bf f}(x:0:z)\big)=L_n(x:0:z)\oplus {\bf f}\big(L_n(x:0:z)\big).
	\end{equation}	
	Therefore, using (\ref{ghjdyhd}) and (\ref{lwesf}), it follows that
	\begin{align*}
		\bigcup_{n=1}^NW_n(\tilde{G})&=\bigcup_{n=1}^N\left\{L_n(x:0:z)\oplus{\bf f}\big(L_n(x:0:z)\big):~(x:0:z)\in \mathbb{P}_{\mathbb{I}\times\{0\}}\right\}\\
		&=\left\{(x:0:z)\oplus{\bf f}(x:0:z):~(x:0:z)\in \mathbb{P}_{\mathbb{I}\times\{0\}}\right\}\\
		&=\tilde{G}.
	\end{align*}
	That is $\tilde{G}$, is also an attractor of the RPIFS $\big\{\mathbb{RP}^2\setminus \mathbb{H}_{e_3}; W_n:~ n=1,2,\ldots,N\big\}$. Hence by the uniqueness of attractor, $G=\tilde{G}$. 
\end{proof}
This completes the proof of {\bf Theorem~\ref{maint11}}.\\

A step by step constructions of a RPFIF on $\mathbb{RP}^2\setminus \mathbb{H}_{e_3}$ and its corresponding self-affine FIF at $z=z_0 \; (\neq 0)$ are illustrated in the following Figures~\ref{initialstep}, \ref{step1}, \ref{step2}, \ref{step3} and \ref{finalstep}.

\begin{figure}[ht!]
	\centering
	\includegraphics[width=4cm, height=4cm]{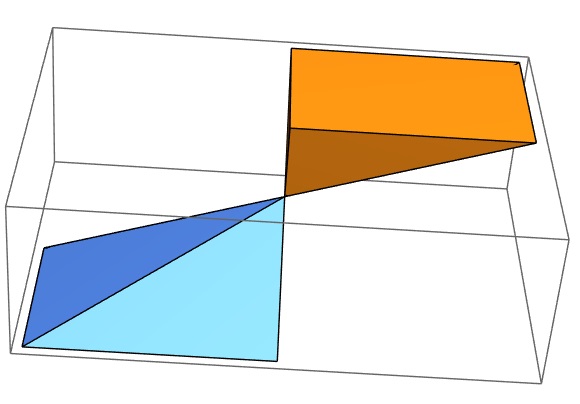}
	\includegraphics[width=3.5cm, height=3.5cm]{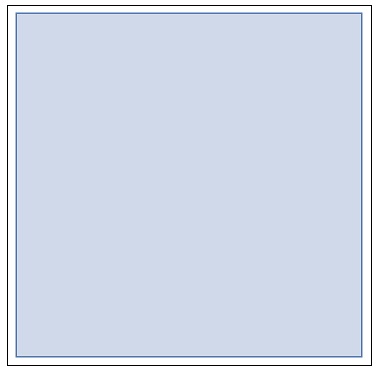}
	\caption{Initial projective rectangle in $\mathbb{RP}^2\setminus \mathbb{H}_{e_3}$ and the corresponding rectangle at $z=z_0 \; (\neq 0)$.} \label{initialstep}
\end{figure}
	\begin{figure}
		\centering
	\includegraphics[width=4cm, height=4cm]{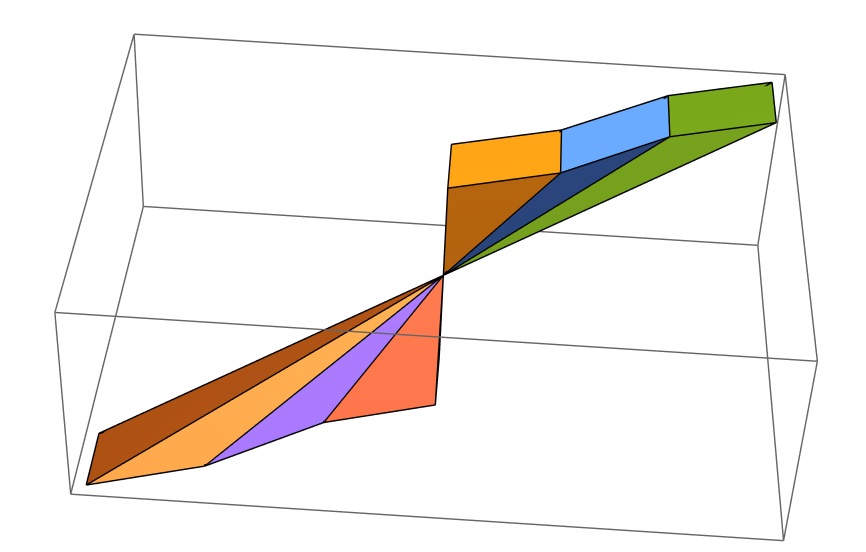}
	\includegraphics[width=3.5cm, height=3.5cm]{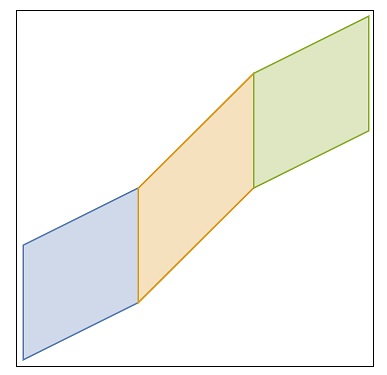}
	\caption{First step of the construction of the  RPFIF and the corresponding FIF.} \label{step1}
	\includegraphics[width=4cm, height=4cm]{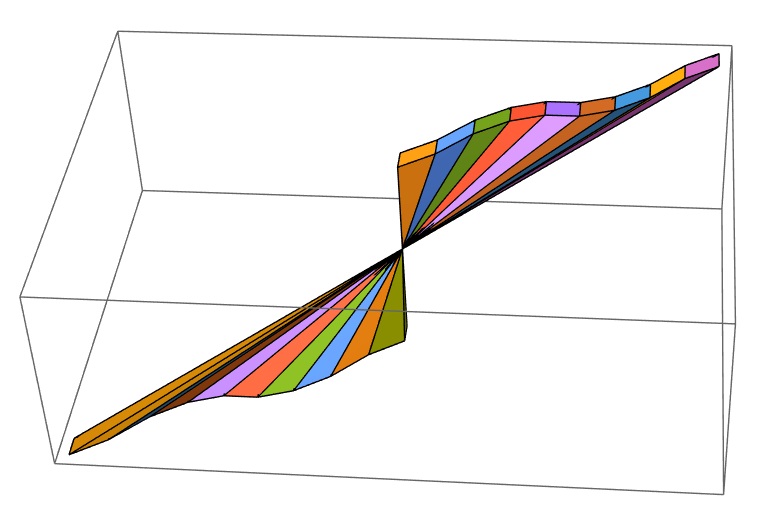}
	\includegraphics[width=3.5cm, height=3.5cm]{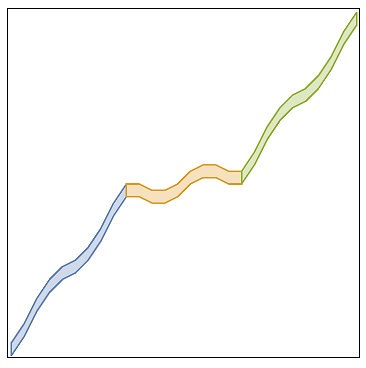}
	\caption{Second step of the construction of the RPFIF and the corresponding FIF.} \label{step2}
	\includegraphics[width=4cm, height=4cm]{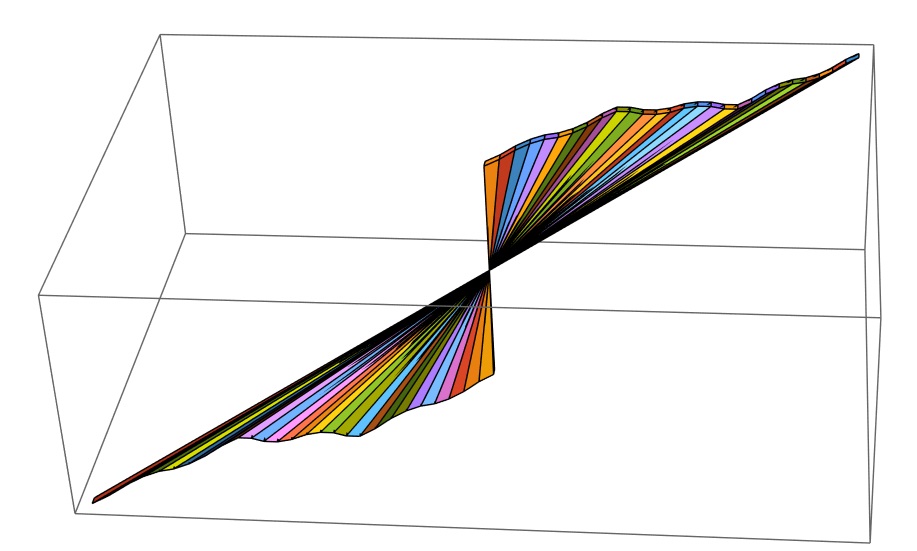}
	\includegraphics[width=3.5cm, height=3.5cm]{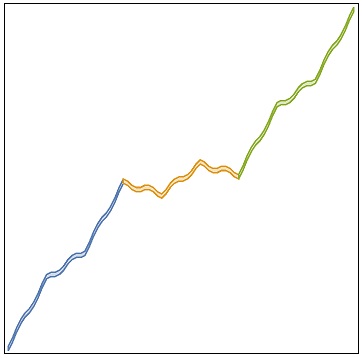}
	\caption{Third step of the construction of the RPFIF and the corresponding FIF.} \label{step3}
	\includegraphics[width=4.5cm, height=4.3cm]{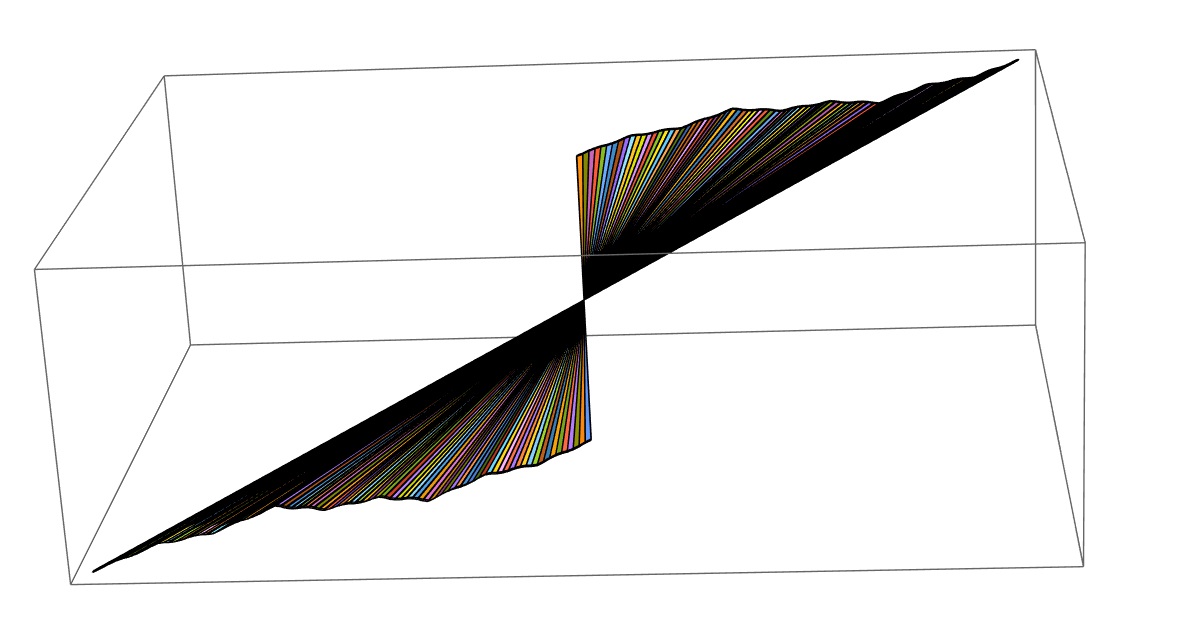}
	\includegraphics[width=4cm, height=4cm]{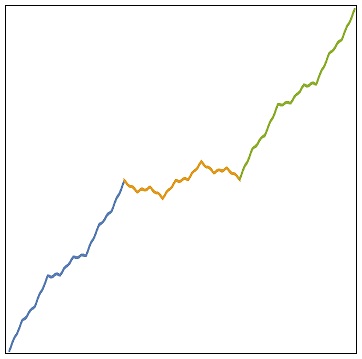}
	\caption{Graphs of the  RPFIF and it's corresponding self-affine FIF at level $z=z_0 \; (\neq 0)$.} \label{finalstep}
\end{figure}
\begin{example}\label{expoftgotrpfif}
	Consider the set of data points   $$\left\{(-2\lambda:\lambda:\lambda),~(-\lambda:-\lambda:\lambda),~(0:\lambda:\lambda),~(\lambda:-\lambda:\lambda),~(2\lambda:\lambda:\lambda)\right\},$$ where $\lambda\neq 0$, and the  scaling factors $d=0.1$, $d=0.3$ and $d=-0.3$ respectively. Then a family of RPFIF is illustrated in  Figure~\ref{DifftypeRPFIF3}, for different scaling factors.
\end{example}

\begin{figure}[!ht]
	\centering
	\includegraphics[width=3cm, height=4cm]{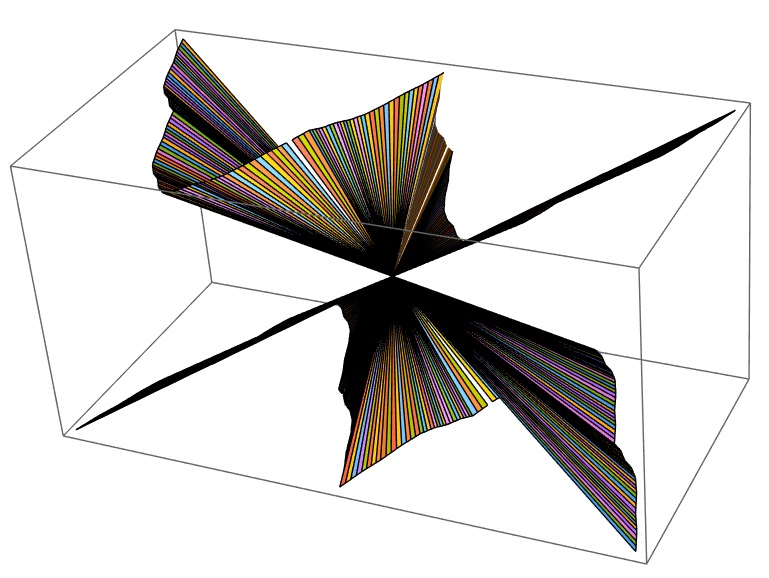}
	\includegraphics[width=3cm, height=4cm]{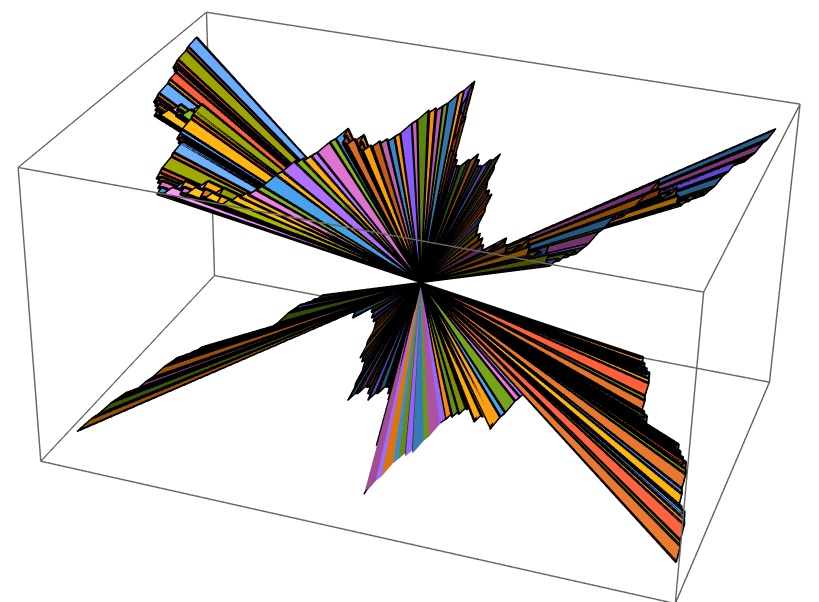}
	\includegraphics[width=3cm, height=4cm]{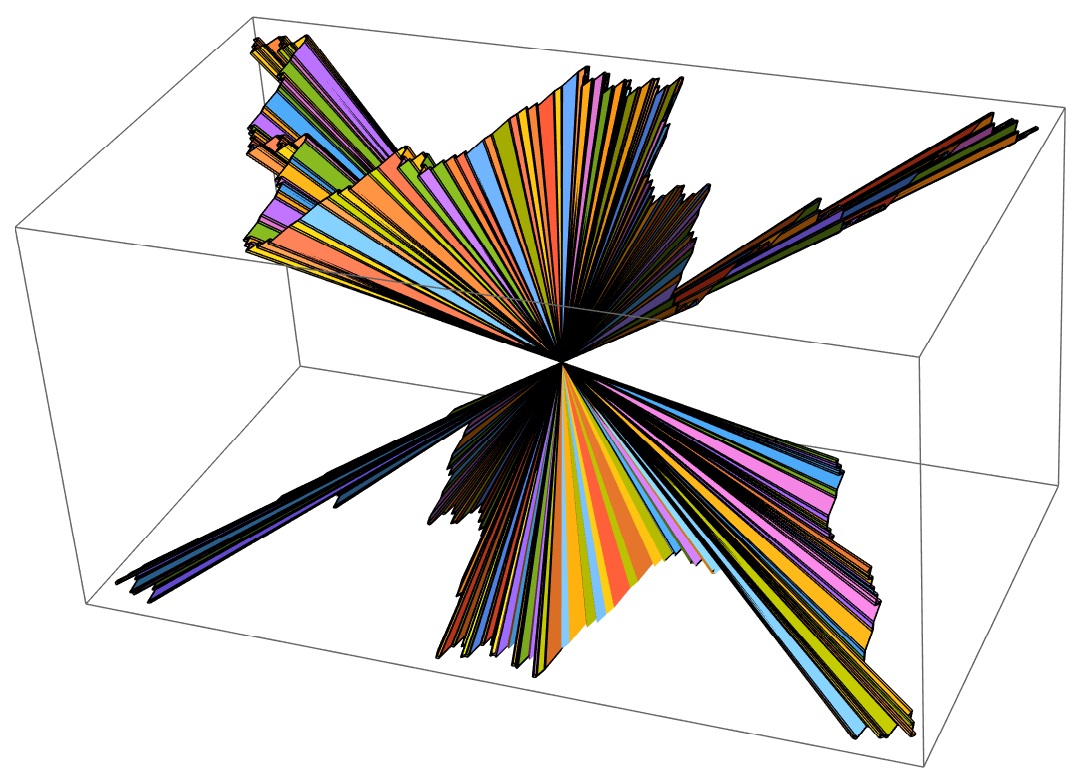}
	\caption{Graphs of the  RPFIFs with scaling factors $d=0.1$, $d=0.3$ and $d=-0.3$ respectively.}\label{DifftypeRPFIF3}
\end{figure}

	\section*{Concluding remarks}
	\begin{remark}
		(Further Extensions). In this article, we considered the space 	$\mathbb{RP}^2$ for notational simplicity only. One may consider projective space $\mathbb{RP}^n$ with more dimensions and deal with the bivariate case. To do that one needs to generalize the operations on the vector space first.\\
		In Section~4, instead of begining with scalar $d_n$, the idea can be extended to the model which considers $d_n$ depending on a variable(s). The prerequisite is to check if these mappings must satisfy some conditions for the fact of working on a projective space. In the ordinary real continuous case, only continuity is required, and there is no need of join-up conditions.\\
		
		Perspective view is the two dimensional replica of a three dimensional figure, where the apparent size of an object decreases as its distance from the viewer point increases. Lenses of camera and the human eye work in the same way, therefore perspective view looks most realistic \cite{Cgbydm}. One can look into the graph of a RPFIF in a different perspective view and estimate the fractal dimension of the corresponding curve which is made by intersection of the graph of the RPFIF with the object plane.
	\end{remark}
	\begin{remark}
		(Motivation for the construction of RPFIF)
		To deal with real world processes which may be irregular in forms, traditional classical interpolants may not provide good approximations. However, the fractal functions which have irregular structure with some degree of self-similarity represent as an alternative to the classical interpolants.
		The non-self-affine fractal analogues $f^{\alpha}$ of any continuous function $f$ form bases for  many standard functions spaces delineating a new field of research referred as fractal approximation theory \cite{ftsbyman}.
		
		However, the more complicated real world phenomena such as tornado, Boy's surfaces, radar, wormhole, etc. may not be well approximated using classical approximant or existing fractal approximant. The Figure~\ref{fig1} is an easy illustration that the projective fractal approximant would be a more suitable approximant rather the existing classical and fractal approximant. 
		\begin{figure}[h!]
			\centering
			\includegraphics[width=7cm, height=5cm]{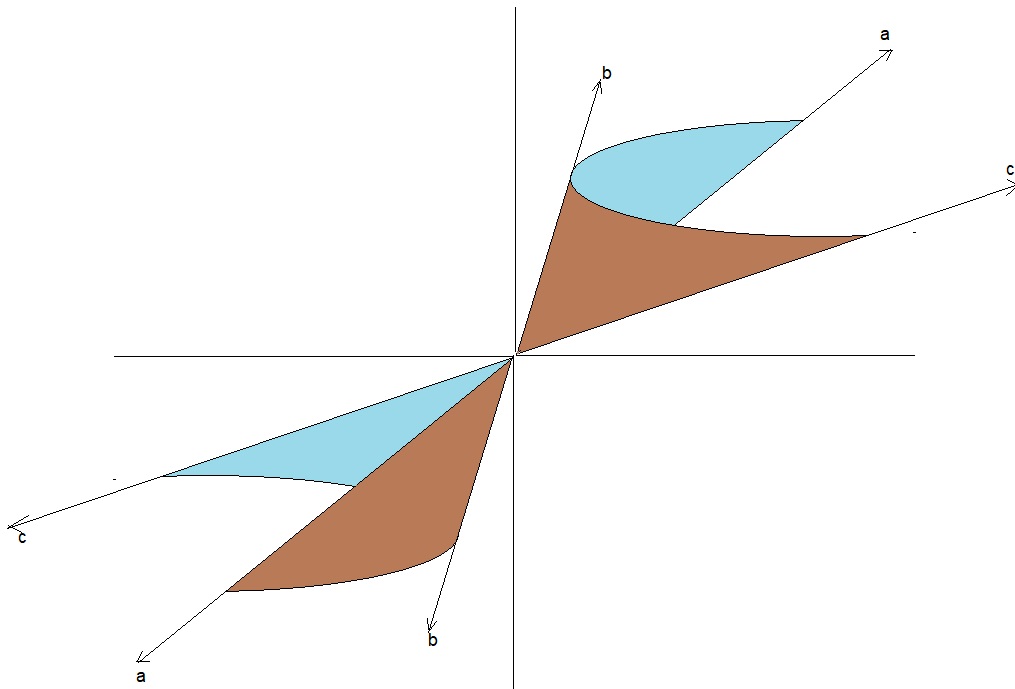}
			\caption{Conic section on the projective plane.} \label{fig1}
		\end{figure}	
	\end{remark}
	\begin{remark}
		(Applications)
		A very real use of projective geometry is given in computer vision. By taking a picture (a 2D perspective of a 3D world) exactly corresponds to a projective transform. On the other hand, fractal transformations generate an image on an attractor  from another image supported on an attractor with a similar IFS structure \cite{Barnsley2013}. So, one may define {\it projective fractal transformations} as application to image processing, pixel changing, camera modeling, etc. Also the projective tiling has many applications in mathematics applied to the real world. So, one may study about the fractal projective tiling on a projective space.

	\end{remark}	
	\begin{remark}
		(Advantages)
		\begin{enumerate}
			\item One of the main advantages working with projective space is that any object at any level can be viewed zooming to ``zero" as well as zooming out to ``infinity".
			\item Self-affine RPFIF displays similarity in projective subintervals.
			\item The level curves of the graph of a RPFIF at each contour are similar. That is the level curve at value $z=a$  is similar to the level curve at  $z=b$ up to contraction.
			\item If we consider the attractor as a subset of $\mathbb{R}^3$, then it is a never ending fractal.
		\end{enumerate}	
	\end{remark}	
	
	{\bf Acknowledgments:} The authors thank Akash Banerjee for many helpful discussions to get the figures.

	\bibliographystyle{siamplain}
	\bibliography{arxivpaperRPFIF}
\end{document}